\documentclass[11pt, letterpaper]{article}
\usepackage{amssymb, url}
\usepackage{amsthm}
\usepackage{amsmath}
\usepackage{verbatim}
\usepackage{graphicx}
\usepackage{caption}
\usepackage{subcaption}
\usepackage[utf8]{inputenc}
\usepackage{fullpage}
\usepackage{setspace}
\usepackage{enumerate}
\usepackage{float}
\usepackage{tikz}
\usetikzlibrary{shapes,arrows,spy,positioning,decorations,calc}
\usetikzlibrary {shapes.geometric}
\usepackage[marginal,splitrule,multiple]{footmisc}
\usepackage{calc}
\usepackage{ifthen}
\usepackage[ruled,vlined,linesnumbered,nofillcomment,noend]{algorithm2e}
\usepackage{stmaryrd}
\usepackage{cite}

\usepackage[pdftex,hypertexnames=false,linktocpage=true]{hyperref}
\hypersetup{colorlinks=true,linkcolor=black,anchorcolor=blue,citecolor=black,filecolor=black,urlcolor=black,bookmarksnumbered=true,pdfview=FitB}

\newcounter{results}[section]
\newcounter{results1}
\newcounter{resultss}

\newtheorem{thm}[results]{Theorem}

\newtheorem{lem}[results]{Lemma}
\newtheorem{prop}[results]{Proposition}
\newtheorem{conj}[results]{Conjecture}

\newtheorem{claim}[results]{Claim}
\newtheorem{case}[results1]{Case}
\newtheorem{subcase}[resultss]{Case}

\newtheorem{defi}[results]{Definition}

\newcommand{\eps}{\varepsilon}

\newcommand{\oururl}{\url{https://lidicky.name/pub/c5frac}}

\if10     
\usepackage[mathlines]{lineno}
\newcommand*\patchAmsMathEnvironmentForLineno[1]{%
  \expandafter\let\csname old#1\expandafter\endcsname\csname #1\endcsname
  \expandafter\let\csname oldend#1\expandafter\endcsname\csname end#1\endcsname
  \renewenvironment{#1}%
     {\linenomath\csname old#1\endcsname}%
     {\csname oldend#1\endcsname\endlinenomath}}%
\newcommand*\patchBothAmsMathEnvironmentsForLineno[1]{%
  \patchAmsMathEnvironmentForLineno{#1}%
  \patchAmsMathEnvironmentForLineno{#1*}}%
\AtBeginDocument{%
\patchBothAmsMathEnvironmentsForLineno{equation}%
\patchBothAmsMathEnvironmentsForLineno{align}%
\patchBothAmsMathEnvironmentsForLineno{flalign}%
\patchBothAmsMathEnvironmentsForLineno{alignat}%
\patchBothAmsMathEnvironmentsForLineno{gather}%
\patchBothAmsMathEnvironmentsForLineno{multline}%
}
\linenumbers
\fi

\title{$C_5$ is almost a fractalizer}

\author{
Bernard Lidick\'{y}\thanks{Department of Mathematics, Iowa State University, Ames, IA, E-mail: {\tt lidicky@iastate.edu}. Research of this author is supported in part by NSF grant DMS-1855653 and DMS-2152490 and Scott Hanna fellowship.}
\and
Connor Mattes \thanks{Department of Mathematical and Statistical Sciences, University of Colorado Denver, E-mail: {\tt connor.mattes@ucdenver.edu}. } \and
Florian Pfender\thanks{Department of Mathematical and Statistical Sciences, University of Colorado Denver, E-mail: {\tt 
Florian.Pfender@ucdenver.edu}. Research is partially supported by NSF grant DMS-1855622 and DMS-2152498.} 
}

\newcommand{\pcycleadd}[2]{   
	\pgfmathsetmacro{\x}{360/(#1+1)}
	\pgfmathsetmacro{\y}{#2-1}
	\pgfmathsetmacro{\z}{#2+1}
	\fill (#2*\x:1.4) circle (2pt);
	\ifthenelse{ #2=1 }{} {\draw[thick] (\y*\x:1)--(#2*\x:1.4);}
	\ifthenelse{ #2=#1}{} {\draw[thick] (#2*\x:1.4)--(\z*\x:1);}
	\draw[thick,dotted] (#2*\x:1)--(#2*\x:1.4);
}

\newcommand{\pcycleaddtwo}[2]{   
	\pgfmathsetmacro{\x}{360/(#1+1)}
	\pgfmathsetmacro{\y}{#2-1}
	\pgfmathsetmacro{\z}{#2+1}
	\fill (#2*\x:1.4) circle (2pt)   (#2*\x:.6) circle (2pt);
	\ifthenelse{ #2=1 }{} {\draw[thick] (#2*\x:1.4)--(\y*\x:1)--(#2*\x:.6);}
	\ifthenelse{ #2=#1}{} {\draw[thick] (#2*\x:1.4)--(\z*\x:1)--(#2*\x:.6);}
	\draw[thick,dotted] (#2*\x:1.4)--(#2*\x:1)--(#2*\x:.6) to[bend left=45] (#2*\x:1.4);
}

\newcommand{\extremal}[5] { 
   \pgfmathsetmacro{\w}{0.16666^#3}
	\node[draw=black, fill=black, minimum size=\w*2.3cm,regular polygon, regular polygon sides = #1] at (#4,#5) {}; 
	\node[draw=black, fill=white, minimum size=\w*1.7cm,regular polygon, regular polygon sides = #1] at (#4,#5) {};
	\node[minimum size=\w*2cm,regular polygon, regular polygon sides = #1] (a) at (#4,#5) {}; 
    	\foreach \y in {1,2,...,#1} {
    		\fill[draw=black, fill=white] (a.corner \y) circle[radius=10pt];
    		\ifthenelse{ #2 = #3 }{} {
			\path (a.corner \y);
			\pgfgetlastxy{\XCoord}{\YCoord}; 
			\extremalb{#1}{#2}{\the\numexpr #3+1\relax}{\XCoord}{\YCoord};
		}
		\ifthenelse{ \y = #1 }{} {\path (a.corner \the\numexpr \y+1\relax); }
    	}
}

\newcommand{\extremalb}[5] { 
   \pgfmathsetmacro{\w}{0.25^#3}
	\node[draw=black, fill=black, minimum size=\w*2.3cm,regular polygon, regular polygon sides = #1] at (#4,#5) {}; 
	\node[draw=black, fill=white, minimum size=\w*1.7cm,regular polygon, regular polygon sides = #1] at (#4,#5) {};
	\node[minimum size=\w*2cm,regular polygon, regular polygon sides = #1] (b) at (#4,#5) {}; 
    	\foreach \y in {1,2,...,#1} {
    		\fill[draw=black, fill=white] (b.corner \y) circle[radius=\w*10pt];
    		\ifthenelse{ #2 = #3 }{} {
			\path (b.corner \y);
			\pgfgetlastxy{\XCoord}{\YCoord}; 
			\extremalc{#1}{#2}{\the\numexpr #3+1\relax}{\XCoord}{\YCoord};
		}
		\ifthenelse{ \y = #1 }{} {\path (b.corner \the\numexpr \y+1\relax); }
    	}
}

\newcommand{\extremalc}[5] { 
   \pgfmathsetmacro{\w}{0.2^#3}
   	\node[draw=black, fill=black, inner sep=0pt,minimum size=\w*2.3cm,regular polygon, regular polygon sides = #1] at (#4,#5) {}; 
	\node[draw=black, fill=white, inner sep=0pt,minimum size=\w*1.7cm,regular polygon, regular polygon sides = #1] at (#4,#5) {};
	\node[inner sep=0pt,minimum size=\w*2cm,regular polygon, regular polygon sides = #1] (c) at (#4,#5) {}; 
    	\foreach \y in {1,2,...,#1} {
    		\fill[draw=black, fill=white] (c.corner \y) circle[radius=\w*10pt];
    		\ifthenelse{ #2 = #3 }{} {
			\path (c.corner \y);
			\pgfgetlastxy{\XCoord}{\YCoord}; 
			\extremal{#1}{#2}{\the\numexpr #3+1\relax}{\XCoord}{\YCoord};
		}
		\ifthenelse{ \y = #1 }{} {\path (c.corner \the\numexpr \y+1\relax); }
    	}
}

\tikzset{
vtx/.style={inner sep=1.5pt, outer sep=0pt, circle, fill=black,draw=black},
}

\begin{document}
\maketitle

\begin{abstract}
We determine the maximum number of induced copies of a 5-cycle in a graph on $n$ vertices  for every $n$. Every extremal construction is a balanced iterated blow-up of the 5-cycle with the possible exception of the smallest level where for $n=8$,  the M\"obius ladder achieves the same number of induced 5-cycles as the blow-up of a 5-cycle on 8 vertices.

This result completes work of Balogh, Hu, Lidick\'y, and Pfender [Eur. J. Comb. 52 (2016)] who proved an asymptotic version of the result. Similarly to their result, we also use the flag algebra method, but we use a new and more sophisticated approach which allows us to extend its use to small graphs.

\noindent\textit{Keywords: inducibility, flag algebras, 5-cycle, fractalizer}\\
\textit{Mathematics Subject Classification: 05C35,  05C38}
\end{abstract}

\section{Introduction}

The {\em inducibility} of a graph $H$ on $k$ vertices is the limit of the maximum density of induced copies of $H$ present in an extremal graph
$G$ on $n$ vertices, where $n$ goes to infinity:
\[
\mbox{ind}(H):=\lim_{n\to\infty}\max_{|G|=n}\frac{|\{\{v_1,\ldots,v_k\}:G[\{v_1,\ldots,v_k\}]\simeq H\}|}{{n\choose k}}.
\]
We say that $G$ is a {\em blow-up of $H$} if either $|H|>|G|$, or if we can get $G$ from $H$ by replacing each vertex $v\in V(H)$ by some non-empty graph $H_v$, and every edge $vw\in E(H)$ by the complete bipartite graph between $H_v$ and $H_w$. If  $|H_v|-|H_w|\le 1$ for any two vertices $v,w\in V(H)$, this is called a {\em balanced blow-up of $H$}. The graph $G$ is an {\em iterated balanced blow-up of $H$} if further every $H_v$ itself is an  iterated balanced blow-up of $H$; see Figure~\ref{fig:C5}. 

Pippenger and Golumbic~\cite{PippengerGolumbic} observe that the iterated balanced blow-ups of $H$ give a lower bound for the inducibility.  In this same paper, they ask for which graphs this bound is sharp, and they conjecture that this bound is sharp for all cycles $C_k$ with $k\ge 5$. 
Balogh, Hu,  Lidick\'y, and Pfender prove the first case $k=5$ in~\cite{C5}, and Brandt, Lidick\'y, and Pfender extend similar methods to the case $k=6$, see~\cite{Brandt}. 
Kr\'{a}l', Norin, and Volec~\cite{KNV2019} give a general upper bound that every $n$-vertex graph has at most $2n^k/k^k$ induced  cycles of length $k$.
In a very recent paper, Blumenthal and Phillips show a result similar to~\cite{C5} for the net graph $N$ on six vertices~\cite{Blumenthal}, the unique graph with degree sequence $(3,3,3,1,1,1)$.
For other results on inducibility of graphs and oriented graphs, see~\cite{liu2020stability,MR1292981,MR4040054,MR3386015,bozyk2020inducibility,tournaments}. 

While inducibility is by definition an asymptotic concept, we are in general interested in the extremal question of maximizing the number of induced copies of a given graph $H$ in a host graph on $n$ vertices, and the extremal graphs. The previous results fall short of a complete answer to this question unless $n=5^k$ or $n=6^k$, respectively. In this paper, we completely answer this question for $H=C_5$, for all $n$.

\begin{figure}
\begin{center}
\scalebox{2}{
\begin{tikzpicture}
\extremal{5}{2}{0}{0}{0}
\end{tikzpicture}
}
\end{center}
\caption{Iterated blow-up of $C_5$.}\label{fig:C5}
\end{figure}
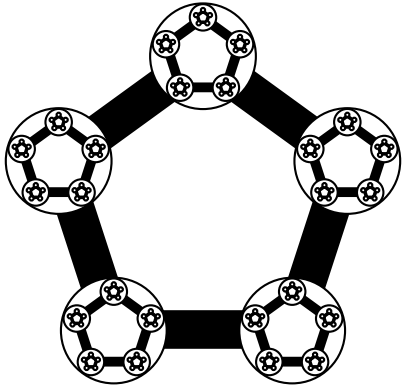

Iterated balanced blow-ups are self-similar much in the same way that fractals are, and so we call a graph $H$ a {\em fractalizer} if its extremal graphs are in fact iterated balanced blow-ups of $H$. 
To make this notion more precise, there are different options to formalize this idea.
\begin{defi}
All of the following properties in some sense formalize the idea of a fractalizer.
\begin{itemize}
\item[(F1)] The iterated balanced blow-ups of $H$ achieve in limit the inducibility of $H$. 
\item[(F2)] There exists an $n_0$ such that for every $n\ge n_0$, some graphs on $n$ vertices maximizing the number of induced copies of $H$ are  balanced blow-ups of $H$.
\item[(F3)] There exists an $n_0$ such that for every $n\ge n_0$, all graphs on $n$ vertices maximizing the number of induced copies of $H$ are  balanced blow-ups of $H$.
\item[(F4)] For every $n$,  an iterated balanced blow-up of $H$ on $n$ vertices maximizes the number of induced copies of $H$.
\item[(F5)] For every $n$, all graphs on $n$ vertices maximizing the number of induced copies of $H$ are iterated balanced blow-ups of $H$.
\end{itemize}
\end{defi}

The following proposition follows straightforward from the definition.
\begin{prop}
For every $H$, (F5) $\Rightarrow$ (F4) $\Rightarrow$ (F2) $\Rightarrow$ (F1) and (F5) $\Rightarrow$ (F3) $\Rightarrow$ (F2) $\Rightarrow$ (F1).
\end{prop} 

In these terms, Pippenger and Golumbic are interested in graphs with (F1). The theorems in~\cite{C5}, \cite{Brandt} and \cite{Blumenthal} imply the stronger notion (F3) for the considered graphs.

The term fractalizer for this concept is due to Fox, Huang and Lee in~\cite{Fox}, and they choose to ask for the strongest notion (F5).
\begin{defi}
A graph $H$ is a {\em fractalizer}, if for every $n$, all graphs on $n$ vertices maximizing the number of induced copies of $H$ are iterated balanced blow-ups of $H$.
\end{defi}

If $H$ is a fractalizer, then its complement is also a fractalizer. Further, each complete and each empty graph is trivially  a fractalizer. Other than these two classes of graphs, no explicit fractalizers are known among simple graphs. We have checked all graphs on up to $7$ vertices, and none of them is a fractalizer. We have ruled out most candidates on $8$ vertices, and $C_8$ is not a fractalizer. We conjecture that $C_9$ is a fractalizer, but our methods are not yet powerful enough to prove it. On the other hand, the main result by Fox, Huang, and Lee~\cite{Fox} implies that almost all graphs are fractalizers: for $n\to\infty$ and constant $p$, a random graph $G_{n,p}$ is almost surely a fractalizer. A similar result is proved independently by Yuster in~\cite{Yuster}. 

The notion of fractalizer can be extended to other structures.
Mubayi and Razborov~\cite{Dhruv} showed that every tournament
on $k \geq 4$ vertices whose edges are colored by $\binom{k}{2}$ distinct colors is a fractalizer in the (F4) sense. They used this to determine the precise number where a certain Ramsey problem transitions from polynomial to exponential growth, settling a conjecture of Erd\H{o}s and Hajnal~\cite{MR0337636} for all $k \geq 4$.

%
%
%
The only non-trivial graph with (F1) on at most $5$ vertices is the $5$-cycle, as all other graphs have constructions with more induced subgraphs in the limit. It has been observed by Michael~\cite{mobius} that for $n=8$, there exist graphs with $8$ induced $5$-cycles other than the balanced blow-ups: the M\"obius ladder on $8$ vertices, i.e.~an $8$-cycle to which we add the $4$ diagonals, and its complement. This implies that for many $n$, there are graphs which match the number of $5$-cycles in the iterated balanced blow-ups. Take for example $n=40$, and consider the balanced blow-up of $H=C_5$ with some of the $H_v$ being  M\"obius ladders. Such a construction extends for all $n$ with $7\cdot 5^k< n < 9\cdot 5^k$ for some $k\in \mathbb{N}$.

The purpose of this paper is two-fold.  We show that $C_5$ has (F4). We do this in a very strong sense, almost showing that $C_5$ is a fractalizer. 
Every extremal graph can differ from an iterated balanced blow-up only at the smallest level, and only in the very limited way described above.
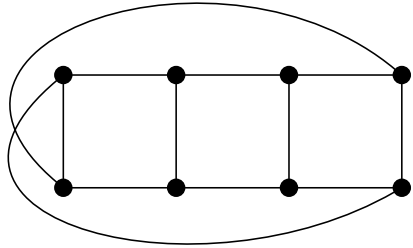
\begin{figure}
\begin{center}
\scalebox{1.5}{
\begin{tikzpicture}
\draw
\foreach \x in {0,1,2,3}{
(\x,0) node[vtx](x0\x){} -- (\x,1) node[vtx](x1\x){}
}
(x00)--(x03)
(x10)--(x13)
(x03) to[out=210,in=220,looseness=1.6] (x10)
(x13) to[out=140,in=140,looseness=1.6] (x00)
;
\end{tikzpicture}
}
\end{center}
\caption{M\"obius ladder on $8$ vertices.}\label{fig:mog}
\end{figure}

\begin{thm}\label{thm_main}
For all $n\ne 8$, all graphs on $n$ vertices maximizing the number
of induced copies of $C_5$ are balanced blow-ups of $C_5$. For $n=8$, the only extremal graphs are the balanced blow-ups of $C_5$, the M\"obius ladder, and its complement. 
Further, the only fractalizers on $5$ vertices are $K_5$ and $\overline{K}_5$.
\end{thm}

As a consequence, this theorem provides a novel proof that the $5$-cycle has (F3) with $n_0=9$, compared to a much larger $n_0$ implied but never determined in~\cite{C5}. We first tried to repeat the arguments in~\cite{C5} to prove Theorem~\ref{thm_main} through some sort of enumeration of small cases, but we quickly realized that this was hopeless. Instead, we find a different and more direct approach that is much more amendable. We still rely heavily on large computations, but the arguments are considerably simpler.

Computations appear in several parts of the proof. First, flag algebra computations are used to establish a key inequality, and this is the only part that requires significant computational resources. Technically, these computations themselves are not part of the proof, but even the certificate in form of a semidefinite matrix is too large to present here. This inequality is then used to show the general structure of the extremal graphs, with a small number of possible defects. These defects are then addressed via stability arguments, yielding more inequalities. For small cases up to $n=1000$, we can then construct all graphs satisfying all inequalities with the help of the computer, and count the cycles. For larger $n$, we first create a continuous model, which we then discretize using a dynamic mesh to show that there are no defects in the construction.

In this write up, we describe all used programs to a point that an interested reader could recreate them, but they are not the main focus of the paper. Oftentimes, we choose simpler programs at the cost of slightly longer running time. While some cases could be checked by hand, and further arguments could reduce some computations, this would not enhance our insight into the problem.
Computer programs used in proofs are available on \href{https://arxiv.org/abs/2102.06773}{arXiv} and at \oururl.

\section{Proof of Theorem~\ref{thm_main}}

The proof proceeds by induction on $n$.  
We use flag algebra calculations to establish an inequality between subgraph densities central to our argument. In this process, we enumerate all graphs with at most $8$ vertices. The extra effort to validate the statement for these graphs is minimal. Therefore, we assume now that $G$ is a graph on $n\ge 9$ vertices, and the statement is true for all smaller graphs.

As $C_5$ is self complementary, we can often simplify our work by using the complement. For this purpose, we interchangeably consider two-colorings of complete graphs with red and blue edges instead of the equivalent model of graphs with edges and non-edges. Note further that every induced red $C_5$ is an induced blue $C_5$ at the same time, so we will often just talk about an induced $C_5$ without specifying the color.
	
We will denote $C(G)$ to be the $5$-cycle density in the graph $G$. In the specific case where $G$ is an iterated balanced blow-up of the $5$-cycle on $n$ vertices, we will denote this quantity by $C(n)$. Note here that all iterated balanced blow-ups of $C_5$ on $n$ vertices have the same number of induced $5$-cycles. If we let $n = 5k+a$, $a,k \in \mathbb{N}, 0 \le a < 5$, then we easily compute 
\begin{equation} \label{eq:blowup}
C(n) = \frac{k^{5-a}(k+1)^{a} + (5-a) {k \choose 5} C(k) + a {k+1 \choose 5}C(k+1)}{{n \choose 5}} .
\end{equation} 
Notice that 
\begin{align}\label{eq:126}
\lim_{n \to \infty}C(n) = \frac{1}{26}.
\end{align}

	
%
%
%

As mentioned above, we will use the flag algebra method to prove a central inequality in Lemma~\ref{funky} below. This is a bit counterintuitive as the method is designed for large graphs, or more precisely, for graph limits, and $G$ has fixed, and possibly small, order. For this reason, we will look at a balanced blow-up $G^*$ of $G$. Flag algebras are then able to give bounds for $G^*$, which we can then use to infer bounds for $G$.

Let $G^*_k$ be the graph which we get by replacing every vertex of $G$ on $n$ vertices by an iterated balanced  blow-up of $C_5$ on $5^k$ vertices, where $k$ is very large, so $|G^*_k| = n 5^k$. Then let $G^*$ be the limit object as $k \to \infty$. This definition ensures that $G^*$ maximizes the density of induced $5$-cycles over all balanced blow-ups of $G$ by the results in~\cite{C5}, but we will not use this fact in our proof. Let $G_v$ for $v \in V(G)$ denote the set of vertices in $G^*$ that are in the blow-up set of $v$.
We can then calculate $C(G^*)$ based on $C(G)$. 
In the following formula we use \eqref{eq:126}.
We further use that every induced $C_5$ in $G^*$ either completely lies in some $G_v$, or intersects five different sets $G_v$ in one vertex each and obtain
\begin{equation} \label{eq:limit} 	
C(G^*) = \frac{n + 26 n (n-1)(n-2)(n-3)(n-4) C(G)}{26 n^5}.
\end{equation}

Similarly as above, in the special case where $G$ is a balanced iterated blow-up of a $5$-cycle on $n$ vertices, we will define $C(n^*) := C(G^*)$. Note that $C(n^*)$ can be calculated explicitly from \eqref{eq:blowup} and \eqref{eq:limit}.  

Let $C^{\bullet \bullet}$ be the class of balanced blow-ups of $C_5$ on $7$ vertices.
There are $6$ different graphs in $C^{\bullet \bullet}$, up to isomorphism, differentiated by the location of the blow-up sets of size two, and by the color of the edges inside the blow-up sets, see Figure~\ref{fig-C**}. 
Let $C^{\bullet \bullet}(G)$ be the combined induced density of $C^{\bullet \bullet}$ in $G$.
For any set $X \subseteq V(G)$ of at most $7$ vertices, let $C^{\bullet \bullet}_X(G)$ denote the density of $7-|X|$ element vertex sets $Y$ disjoint from $X$ such that $G[X \cup Y]$ is isomorphic to a graph in $C^{\bullet \bullet}$. 

\begin{figure}
\begin{center}
\newcommand{\base}{
\draw \foreach \x  in {0,...,4}{
(90+\x*72:1) node[vtx](x\x){}
(90+\x*72:1.3) coordinate(y\x)
(90+\x*72:1.4) coordinate(z\x)
}
(x0)--(x1)--(x2)--(x3)--(x4)--(x0)
;
}
\begin{tikzpicture} \base \draw (y4) node[vtx]{}  (x3)--(y4)--(x0)   (y1) node[vtx]{} (x0)--(y1)--(x2);  \end{tikzpicture}
\begin{tikzpicture} \base \draw (y4) node[vtx]{}  (x3)--(y4)--(x0)   (x1) -- (y1) node[vtx]{} (x0)--(y1)--(x2);  \end{tikzpicture}
\begin{tikzpicture} \base \draw (x4)-- (y4) node[vtx]{}  (x3)--(y4)--(x0)    (x1) --  (y1) node[vtx]{} (x0)--(y1)--(x2);  \end{tikzpicture}\\
\begin{tikzpicture} \base \draw (y3) node[vtx]{}  (x2)--(y3)--(x4)   (y2) node[vtx]{} (x1)--(y2)--(x3) (y2)--(y3);  \end{tikzpicture}
\begin{tikzpicture} \base \draw (x3)--(y3) node[vtx]{}  (x2)--(y3)--(x4)   (y2) node[vtx]{} (x1)--(y2)--(x3) (y2)--(y3);  \end{tikzpicture}
\begin{tikzpicture} \base \draw (x3)--(y3) node[vtx]{}  (x2)--(y3)--(x4)   (x2)--(y2) node[vtx]{} (x1)--(y2)--(x3) (y2)--(y3);  \end{tikzpicture}
\end{center}
\caption{The $6$ different graphs in $C^{\bullet \bullet}$, only red edges are depicted.}\label{fig-C**}
\end{figure}
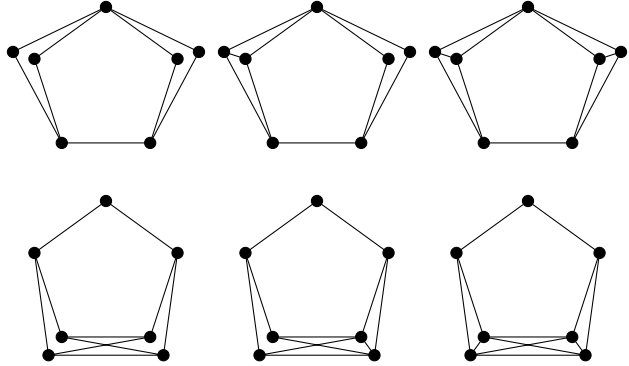

We bound $C(G)$ in terms of $C^{\bullet \bullet}(G)$ using the flag algebra method. We defer the proof of this key lemma to Section~\ref{sec:lem}.
%

\begin{lem}\label{funky}
	For every graph $G$ with $C(G^*) > 0.03$,
	\[ C^{\bullet \bullet}(G^*) \ge -0.175431374077117 + 8.75407592662244 \, C(G^*) .\]
\end{lem}


Assume from now on that $G$ is extremal, i.e. $G$ maximizes the number of induced $5$-cycles over all graphs on $n$ vertices. In particular, $C(G^*)\ge C(n^*)$. We compute $C(n^*)$ explicitly for $n< 100$, and observe that $C(n^*)>0.03$. For $n\ge 100$, we have
\[
C(n^*)>\left\lfloor\frac{n}{5}\right\rfloor^5\frac{5!}{n^5}\ge 5!\left(\frac{n-4}{5n}\right)^5\ge 5!\left( \frac{96}{500}\right)^5>0.031,
\]
so Lemma~\ref{funky} applies to $G$. Our goal is to show that the top level of $G$ is a blow-up of $C_5$, i.e. $V(G)$ can be partitioned into five non-empty parts $X_1,X_2,X_3,X_4,X_5$, 
such that all edges between $X_i$ and $X_j$ are red if $|i-j|\in \{1,4\}$, and blue if $|i-j|\in\{2,3\}$. Towards this, for any partition $V(G)=X_1\cup X_2\cup X_3\cup X_4\cup X_5$, call an edge {\em funky} if it has the wrong color according to this partition. We will denote the set of funky edges by $E_f$, and the number of funky edges incident to a vertex $v$ by $d_f(v)$. Let $x_i:=\frac1n |X_i|$ be the normalized sizes of the parts, and let $f{n\choose 2}=|E_f|$ be the number of funky edges. A partition is more desirable if it contains more edges between different parts which are not funky. Note that our desired balanced partition maximizes this quantity for a given $n$. While we cannot guarantee this perfect partition at this point, we can show a lower bound.

\begin{lem}\label{numFunky}
	There exists some partition of $V(G)$ into $X_1,\dots,X_5$ such that,
	
	\[  \sum_{1\le i < j \le 5} x_i x_j - \frac{{n \choose 2}}{n^2} f \ge \frac{2(-0.175431374077117 + 8.75407592662244 \, C(n^*))}{21 \, C(n^*)} . \] 

\end{lem}

\begin{proof}
Let $Z$ be a set of five vertices in $V(G^*)$ inducing a $C_5$ such that  $C^{\bullet \bullet}_Z(G^*)$ is maximized. As $C_5$ is not a blow-up of any graph $H$ with $2\le |H|\le 4$, there are two cases to consider. Either $Z\subset G_v$ for some $v\in G$, or $|Z\cap G_v|\le 1$ for all $v\in V(G)$, and the vertices $v\in V(G)$ with $|Z\cap G_v|=1$ induce a $C_5$ in $G$. We claim the later is true.

If $Z\subset G_v$, then any vertex set $Y$ such that $Y\cup Z$ induces a graph in $C^{\bullet \bullet}$ must also be in $G_v$. Thus, $C^{\bullet \bullet}_Z(G^*)\le \frac1{n^2}$. On the other hand, as $G$ contains $5$-cycles, we can find a $Z$ with $|Z\cap G_{v_i}|=1$ for $1\le i\le 5$, and $v_1v_2v_3v_4v_5v_1$ an induced $5$-cycle in $G$. Then $Y\cup Z$ induces a graph in $C^{\bullet \bullet}$ for any choice of $Y$ intersecting exactly two of the $G_{v_i}$, and thus  $C^{\bullet \bullet}_Z(G^*)\ge \frac{20}{n^2}$, proving that $Z \not\subset G_v$ for any $v$.
	
	As $Z$ maximizes $C^{\bullet \bullet}_Z(G^*)$, we know that $C^{\bullet \bullet}_Z(G^*)$ is greater than or equal to the average over all sets inducing a $5$-cycle in $G^*$. For any graph in $C^{\bullet \bullet}$, exactly $4$ of the $21$ subgraphs on $5$ vertices are $5$-cycles. Therefore, 
	\begin{align*} C^{\bullet \bullet}_Z(G^*) &\ge \frac{4 \, C^{\bullet \bullet}(G^*)}{21 \, C(G^*)} \\
	&\ge \frac{4 ( -0.175431374077117 + 8.75407592662244 \,C(G^*))}{21 \,C(G^*)} \text{\hskip 2em by Lemma \ref{funky},}\\
	&\ge \frac{4 ( -0.175431374077117 + 8.75407592662244 \,C(n^*))}{21 \,C(n^*)}, 
	\end{align*}
where the last inequality is true since $C(G^*)\ge C(n^*)$, and the function is monotone increasing.	
	
Now partition $V(G) = X_1 \cup \dots \cup X_5 $ according to $Z$, that is, if $v\in V(G)$ and $\{v_1,v_2,v_3,v_4,v_5\}\setminus \{v_i\}\cup \{v\}$ is a $5$-cycle, then $v\in X_i$. Note that this rule assigns $v$ to at most one $X_i$. The remaining vertices are assigned to the $X_i$ arbitrarily. Observe that for $v^*\in G_v$, $w^*\in G_w$, $Z\cup \{v^*,w^*\}$ induces in $G^*$ a graph in $C^{\bullet \bullet}$ if and only if both $v$ and $w$ are assigned to different $X_i$ by the rule, and the edge $vw$ is not funky. 
Therefore,
\begin{align*}
\frac{ \sum_{i \ne j } |X_i|\,|X_j| - 2{n \choose 2} f}{n^2} \ge C^{\bullet \bullet}_Z(G^*),
\end{align*}
and the lemma follows.
\end{proof}


The following technical lemma is helpful in creating the mathematical programs used in some of the remaining claims.

\begin{lem}\label{lem}
	Let $G$ be a graph on $n$ vertices, and let $X\subset V(G)$.
	\begin{enumerate}
		\item If $|X|=1$, then $X=\{x\}$ is contained in at most $\frac{r^2b^2}{16}\le \left( \frac{n-1}{4}\right)^4$ copies of an induced $C_5$, where $r$ and $b$ are the numbers of red and blue neighbors of $x$, respectively.
		\item If $|X|=2$, then $X$ is contained in at most $\left( \frac{n-2}{3}\right)^3$ copies of an induced  $C_5$.
		\item If $|X|=3$, then $X$ is contained in at most $\left( \frac{n-3}{2}\right)^2$ copies of an induced  $C_5$.
	\end{enumerate}
\end{lem}

\begin{proof}
	To see the second and third statement, notice that the edges in $X$, and the edges from any vertex in $V(G) - X$ to $X$ completely determine where on a $C_5$ that vertex can lie, or if it can lie on a $C_5$ at all. For instance, if $X = \{w_1,w_2\}$, $w_1w_2$ is red, and $w_1w_2w_3w_4w_5w_1$ is a red cycle, then for each $w_i$, $3\le i\le 5$, the colors of $(w_1w_i,w_2w_i)$ are different. Therefore we can maximize the number of $5$-cycles by partitioning the vertices in $V(G)\setminus X$ into two (or three) equal classes with the edges colored these ways. 
	
	To see the first statement, notice that every $C_5$ containing $x$ has exactly two red and two blue neighbors of $x$. 
	For every red neighbor $v$ and blue neighbor $w$, let 
	\[
	a(v,w)=\begin{cases} 1,&\mbox{ if $vw$ is red},\\ 0,&\mbox{ if $vw$ is blue}.\end{cases}
	\]
	Denote $|a(.,w)|$ as the number of ones in $a(.,w)$, that is the number of red neighbors shared between $w$ and $x$. For $u,v$ red neighbors of $x$, let $h(u,v)$ be the Hamming distance of the two vectors $a(u,.),a(v,.)\in \{0,1\}^b$, that is the number of coordinates where $a(u,.)$ and $a(v,.)$ differ. This quantity is important as every $C_5$ containing $\{x,u,v\}$ must contain one vertex $w$ with $a(u,w)=1-a(v,w)=0$ and one vertex $y$ with
	$a(u,y)=1-a(v,y)=1$. In particular, there can be at most $\frac{h(u,v)^2}{4}$ $5$-cycles containing $\{x,u,v\}$. Therefore the number of $5$-cycles is at most
	\begin{align*} 
		\frac{1}{4} \sum_{xu, xv \text{ red }} h(u,v)^2 & \le \frac{\max_{xu, xv \text{ red }} h(u,v)}{4} \sum_{xu, xv \text{ red }} h(u,v) \\
		& \le \frac{b}{4} \sum_{xu, xv \text{ red }} h(u,v) \\
		&= \frac{b}{4} \sum_{xw \text{ blue }} |a(.,w)|(r - |a(.,w)|) \\
		&\le \frac{b^2r^2}{16}.
	\end{align*}

\end{proof}

We are now ready to show that in a partition $V(G)=X_1\cup X_2\cup X_3\cup X_4\cup X_5$ maximizing the number of non-funky edges between parts, there are no funky edges.
We split the argument into two parts, depending on the size of $n$. 

\begin{case} \label{medium}
	$9 \le n \le 1000$:
\end{case}

	We first change the color of all funky edges to create a graph $G_1$ without funky edges, where we also change the graphs inside the $X_i$ to iterated balanced blow-ups of $C_5$. The density of $5$-cycles in $G_1$ is then easily calculated as
	\[
	C(G_1)= \frac{120x_1x_2x_3x_4x_5n^5+\sum_i x_in(x_in-1)(x_in-2)(x_in-3)(x_in-4)C(x_in)}{n(n-1)(n-2)(n-3)(n-4)}.
	\]
	
	Furthermore, we provide generous bounds on the number of $5$-cycles created and destroyed going from $G$ to $G_1$ (see Claims \ref{count1}, \ref{count2}, and \ref{count3}). This together with the number of cycles in $G_1$ allows us to bound the number of $5$-cycles in $G$ without directly counting them. 
	
	We then create an integer program $(P)$, for a fixed number of vertices, with an objective function of the difference between the bound on the number of $5$-cycles in $G$ discussed above and the number $C(n)$ of $5$-cycles in the balanced iterated blow-up on the same number of vertices. We then iterate through all possible sizes of the $X_i$ for $9$ to $1000$ vertices. In this way, the program yields a contradiction for most choices of the $X_i$. The few remaining cases only appear on a relatively small number of vertices. This allows us to check these cases by a brute force method. 
	
	To create our program $(P)$, let $y_1, \dots, y_5$ be a permutation of the $x_i$'s such that $y_1 \ge \dots \ge y_5$. 
Recall that $f  := |E_f|/\binom{n}{2}$ is the scaled number of funky edges.	
	If $f=0$, we are done, so assume that $f>0$. Let
	\[
	d=\frac{1}{f{n\choose 2}n}\sum_{xy\in E_f}(d_f(x)+d_f(y)-2)
	\]
	be the average number of funky edges incident to a funky edge, divided by $n$. 
	
	\begin{claim}\label{count1}
		The graph $G$ contains at most
		\[\frac12 f{n\choose 2}\left(f{n\choose 2}-dn-1 \right)\left(\left( y_1+y_2+\frac12(y_3+y_4+y_5)\right)n-2\right)\]
		$5$-cycles which contain at least two non-incident funky edges. 
	\end{claim}
	
	\begin{proof}
		
		Pick two non-incident funky edges. In other words, we pick a funky edge, and then pick another funky edge not incident to the first one, and then multiply this count by $\frac12$ because we counted every pair of edges twice. We can do this in
		\begin{equation}\label{noninc}
		\frac12\sum_{xy\in E_f}\left(f{n\choose 2}-d_f(x)-d_f(y)+1\right)=\frac12 f{n\choose 2}\left(f{n\choose 2}-dn-1 \right)
		\end{equation}
		ways, where the ``$+1$" comes from double counting the edge $xy$ in both $d_f(x)$ and $d_f(y)$. 
		
		The four vertices, let us call them $\{w,x,y,z\}$, spanning the pair of funky edges  must induce a red (and a blue) $P_4$, as otherwise they cannot induce a $C_5$ with a fifth vertex. Without loss of generality assume $wx,xy,yz$ are the red edges inducing the $P_4$. To count the $5$-cycles we must then pick a $5^{\mathrm{th}}$ vertex (call this vertex $v$) such that $vw$ and $vz$ are red, and $vx$ and $vy$ are blue. Note that with the proper combination of funky, non-funky, and edges within the $X_i$s, $v$ can be an element of any $X_i$. However, if any edge between $v$ and $\{w,x,y,z\}$ is funky, then this $C_5$ contains at least two pairs of non-incident funky edges. As a consequence, our counting strategy of first choosing a pair of funky edges, and then adding a fifth vertex, will count this $5$-cycle at least twice. To make up for this, we can add a factor of $\frac12$ to the number of such $5$-cycles. Therefore, in order to prove the claim, it suffices to show that no matter the location of $\{w,x,y,z\}$, there are at most two sets $X_i$, such that we can have $v\in X_i$ and no funky edge between $v$ and $\{w,x,y,z\}$.
		

If $wx$ is funky, we may assume by symmetry that $w \in X_1$ and $x \in X_3$. In this case the only two sets where $v$ may lie so 
that neither the red edge $vw$ nor the blue edge $vx$ is funky,
are $X_1$ and $X_5$. Similarly if $xw$ is not funky we may assume by symmetry that $x \in X_1, w \in X_2$. In this case the only sets that $v$ can be in so that neither $vw$ nor $vx$ are funky are $X_2$ and $X_5$. 

Hence the number of choices for $v$ to complete the $C_5$ is at most 
\[
\left( y_1+y_2+\frac12(y_3+y_4+y_5)\right)n-2,
\]
where $-2$ comes from  $v \not\in \{w,x,y,z\}$.
Multiplying this with \eqref{noninc} finishes the proof of the claim. 
	\end{proof}
	
	\begin{claim}\label{count2} The graph $G$ contains at most
		\[ \frac{9}{32}(dn+2)f{n\choose 2}y_1^2 n^2 \]
		$5$-cycles with at least one funky edge, but without two non-incident funky edges. 
	\end{claim}
	
	\begin{proof}
		Note that no $C_5$ in $G$ can contain exactly one funky edge. If a $C_5$ does not contain two non-incident funky edges, then 
		either all funky edges are incident to a single vertex of the cycle, or there are exactly three funky edges forming a triangle.

		Let $v$ be a vertex incident to at least two funky edges in the $C_5$ we want to count, and say $v \in X_1$. If the funky edges in the $C_5$ we want to count form a triangle, note that this triangle must contain edges of both colors as $C_5$ does not contain a monochromatic triangle. In this case, choose $v$ to be a vertex incident to funky edges of both colors.
		We break the count up into cases based on the colors of funky edges incident to $v$, each of which will correspond to a term in a sum. Illustrations are provided in Figure~\ref{fig:count2}.
		
		Case I: $v$ is incident to at least two red funky edges in the $C_5$, say to vertices $u,w\in X_3\cup X_4$. We know that $u$ and $w$ must be in the same set as otherwise the three vertices induce a red triangle, or $uw$ is funky and we would have chosen a different vertex as $v$. By symmetry say $u,w \in X_{3}$. The other two vertices in a $C_5$ must each have exactly one red and one blue edge to $\{u,w\}$, which, without funky edges not incident to $v$, can only happen if they are also in $X_{3}$. We can then directly apply part $1.$ of Lemma \ref{lem} to count at most $\frac{(r_f(v))^2}{4} \cdot\frac{(y_1n)^2}{4}$ $5$-cycles for each such $v\in V$.
		
		Case II: $v$ has at least two blue funky edges. Similarly to Case I, by applying Lemma \ref{lem} we count at most $\frac{(b_f(v))^2}{4}\cdot\frac{y_1^2n^2}{4}$ $5$-cycles for each such $v \in V$. 
		
		Case III: $v$ has exactly one blue funky edge $vu$ and one funky red edge $vw$. The edge $uw$ may be either funky or not. Then $u,v,w$ are in different sets $X_i$, and they span a red or blue $P_3$. By symmetry, we may assume that it is a red $P_3$ $vwu$, with the red cycle being $vwuxyv$. As $uv$ is funky and blue, we may again assume by symmetry that $u\in X_2$. We then have two subcases. First, subcase IIIa: $w\in X_3$. Then $y\in X_5$ as both $uy$ and $wy$ are blue, and $x\in X_1$ as both $ux$ and $xy$ are red. Similarily we have subcase IIIb: $w\in X_4$. Then $x\in X_1$ as $ux$ is red (so $x\notin X_4\cup X_5$), $vx$ is blue (so $x\notin X_2$), and $wx$ is blue (so $x\notin X_3$). Similarly, $y\in X_2$.
		
%
%

\begin{figure}
\begin{center}
\begin{tikzpicture}[scale=0.8]
\draw[line width=8pt] (0,0) node[regular polygon,regular polygon sides=5,minimum size=4.15cm,draw]{};
\foreach \x in {1,2,3,4,5}{
\draw[fill=gray!50!white](90+\x*72:2)coordinate(x\x) circle(0.8 cm);
\draw(90+\x*72:3.15) node{$X_\x$};
}
\draw[thick]
(x1) node[vtx,label=left:{\color{black}$v$}](v){}
(x3) ++(50:0.3) node[vtx,label=right:{\color{black}$w$}](w){}
(x3) ++(230:0.3) node[vtx,label=below:{\color{black}$u$}](u){}
(x3) ++(290:0.3) node[vtx](a){}
(x3) ++(350:0.3) node[vtx](b){}
(u) -- (v) -- (w)  -- (b) -- (a) -- (u)
;
\draw (0,-3) node{Case I};
\end{tikzpicture}
\hskip 1em
\begin{tikzpicture}[scale=0.8]
\draw[line width=8pt] (0,0) node[regular polygon,regular polygon sides=5,minimum size=4.15cm,draw]{};
\foreach \x in {1,2,3,4,5}{
\draw[fill=gray!50!white](90+\x*72:2)coordinate(x\x) circle(0.8 cm);
\draw(90+\x*72:3.15) node{$X_\x$};
}
\draw[thick]
(x1) node[vtx](v){}
(x2) ++(10:0.3) node[vtx](w){}
(x2) ++(190:0.3) node[vtx](u){}
(x2) ++(250:0.3) node[vtx](a){}
(x2) ++(310:0.3) node[vtx](b){}
(u) -- (v) -- (w)  -- (b) -- (a) -- (u)
;
\draw (0,-3) node{Case II};
\end{tikzpicture}
\vskip 1em
\begin{tikzpicture}[scale=0.8]
\draw[line width=8pt] (0,0) node[regular polygon,regular polygon sides=5,minimum size=4.15cm,draw]{};
\foreach \x in {1,2,3,4,5}{
\draw[fill=gray!50!white](90+\x*72:2)coordinate(x\x) circle(0.8 cm);
\draw(90+\x*72:3.15) node{$X_\x$};
}
\draw[thick]
(x1)++(0.2,0)node[vtx,label=right:$v$](v){}
(x2) node[vtx,label=below:$u$](u){}
(x3)  node[vtx,label=right:$w$](w){}
(x5)  node[vtx,label=above:$y$](y){}
(x1)++(-0.2,0)  node[vtx,label=above:$x$](x){}
(v)--(w)--(u)--(x)--(y)--(v)
;
\draw (0,-3) node{Case IIIa};
\end{tikzpicture}
\hskip 1em
\begin{tikzpicture}[scale=0.8]
\draw[line width=8pt] (0,0) node[regular polygon,regular polygon sides=5,minimum size=4.15cm,draw]{};
\foreach \x in {1,2,3,4,5}{
\draw[fill=gray!50!white](90+\x*72:2)coordinate(x\x) circle(0.8 cm);
\draw(90+\x*72:3.15) node{$X_\x$};
}
\draw[thick]
(x1)++(0.2,0)node[vtx,label=above:$v$](v){}
(x2)++(-0.1,-0.4) node[vtx,label=left:$u$](u){}
(x4)  node[vtx,label=right:$w$](w){}
(x2)++(0,0.1)  node[vtx,label=above right:$y$](y){}
(x1)++(-0.2,0)  node[vtx,label=above:$x$](x){}
(v)--(w)--(u)--(x)--(y)--(v)
;
\draw (0,-3) node{Case IIIb};
\end{tikzpicture}
\end{center} 
\caption{Cases where $v$ is incident with two funky edges from Claim~\ref{count2}. Only red edges are depicted.}\label{fig:count2}
\end{figure}
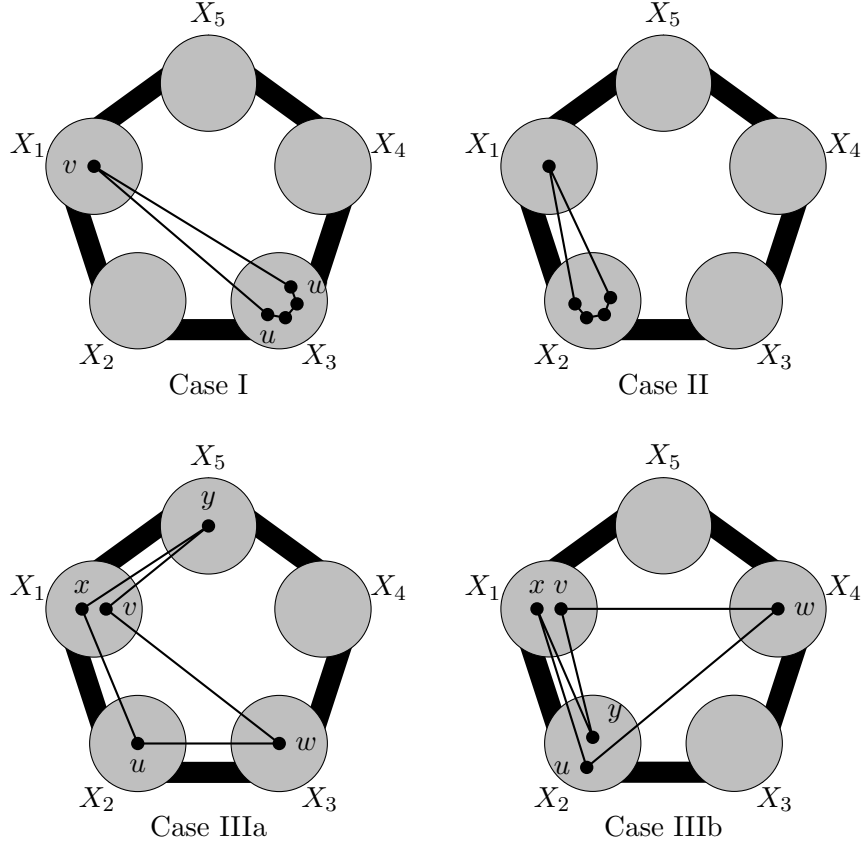
		
		Therefore for any choice of funky edges in this case, the two sets for $x$ and $y$ are determined, and they are different. This gives us an upper bound of $r_f(v)b_f(v)y_1^2n^2$ $5$-cycles of this type containing $v$.
		
		Putting the three cases together, there are at most
		\begin{align*}
		&\sum_{v\in V} \left(  \frac{r_f(v)^2}{4}\frac{y_1^2 n^2}{4}+ \frac{b_f(v)^2}{4}\frac{y_1^2n^2}{4} +r_f(v)b_f(v)y_1^2n^2\right)\\
		&=y_1^2 \sum_{v\in V} \left(\left(  \frac{r_f(v)n}{4}+ \frac{b_f(v)n}{4}\right)^2 +\frac{7}{8}r_f(v)b_f(v)n^2\right)\\
		&\le y_1^2 \sum_{v\in V} \left(\left(  \frac{d_f(v)n}{4}\right)^2+\frac78\left(  \frac{d_f(v)n}{2}\right)^2\right)\\
		&=y_1^2 \sum_{v\in V} \frac{9}{32}\left(  d_f(v)n\right)^2\\
		&=y_1^2 \frac{9}{32}\sum_{vw\in E_f}(d_f(v)+d_f(w))n^2\\
		&=\frac{9}{32}(dn+2)f{n\choose 2}y_1^2n^2
		\end{align*}
		$5$-cycles in $G$ containing funky edges but no pair of non-incident funky edges. 
		
	\end{proof}
	
Now we are counting the new $5$-cycles when switching from $G$ to $G_1$.
	
	\begin{claim}\label{count3}
		
	The graph $G_1$ contains at least 
		\[ f{n\choose 2}n^3\left(  y_3y_4y_5-\frac38 dy_3y_4-\frac18fy_3 \right)  \]
		$5$-cycles whose vertex set spans at least one funky edge in $G$. 
		
	\end{claim}
	
	\begin{proof}
	Note that the new $5$-cycles are exactly the vertex sets $\{v_1,v_2,v_3,v_4,v_5\}$ with $v_i\in X_i$ which span at least one funky edge in $G$. We count these cycles using inclusion and exclusion principle by counting pairs $(F,C)$, where $F$ is a set of funky edges in $G$, and $C$ is a $5$ cycle in $G_1$ containing the vertices of $F$.

We start by counting pairs $(\{vw\},C)$, where $vw$ is a funky edge in $G$. First we pick a vertex $v$, then a funky neighbor $w$ from the $d_f(v)$ choices, and then one vertex each from the three parts we have not yet used, which gives us at least $y_3y_4y_5n^3$ choices. Summing up over all choices of $v$, this double counts the pairs, as we can reverse the roles of $v$ and $w$, and we multiply by $\frac12$ to get the first term of the bound
		\begin{align}\label{base}
		\sum_{v\in V}  \frac12d_f(v)y_3y_4y_5n^3.
		\end{align}
This would be the number of new cycles if every new cycle contained exactly one funky edge. But new cycles with $2\le r\le 10$ funky edges are counted $r$ times by this bound, so we have to carefully correct for this.

In the next step, we are counting pairs $(\{vw,xy\},C)$, with $vw, xy$ distinct funky edges in $G$. First, we are counting cycles with $v=x$. For a vertex $v$, there are at most ${4 \choose 2} \cdot \left( d_f(v)/4 \right)^2 =  \frac{3}{8}d_f(v)^2$ ways to pick $\{w,y\}$ from two different sets, with equality if $v$ sends the same number of funky edges to each of the four parts. Then, the remaining two vertices for $C$ are picked from the two remaining sets. As we are correcting for the double count in~\eqref{base}, this is maximized if these two last sets have sizes $y_3n$ and $y_4n$.

Next, we are counting cycles with $vw$ and $xy$  non-incident, i.e.~the funky edges intersect four parts. We claim that there are at  most $\frac{\left(f{n\choose 2}\right)^2}{4}$ pairs of funky edges intersecting four parts. Consider the graph with vertex set $E_f$, and two members of $E_f$ are adjacent if they intersect a common $X_i$. As $K_5$ has matching number $2$, this graph has independence number at most $2$. By Mantel's Theorem this graph has at most $\frac{|E_f|^2}{4}$ non-edges, which correspond exactly to pairs of funky edges intersecting four parts in $G$.

For every such pair of funky edges, we choose a fifth vertex in the remaining part to complete a new $C_5$ in $G_1$. As we are correcting for the double count in~\eqref{base}, this is maximized if this last set has sizes $y_3n$.

If we subtract the count of pairs $(\{vw,xy\},C)$ from \eqref{base}, every cycle with $r$ funky edges is counted $r-{r\choose2}\le 1$ times. In total, this gives us a lower bound for new $5$-cycles in $G_1$:
\begin{align*}
&y_3y_4y_5n^3\sum_{v\in V}  \frac12d_f(v)-\frac{3}{8}y_3y_4n^2\sum_{v\in V}d_f(v)^2 - \frac{\left(f{n\choose 2}\right)^2}{4}y_3n\\
&= y_3y_4y_5n^3f{n\choose 2}-\frac{3}{8}y_3y_4n^2\sum_{vw\in E_f}(d_f(v)+d_f(w))- \frac{\left(f{n\choose 2}\right)^2}{4}y_3n\\
&= y_3y_4y_5n^3f{n\choose 2}-\frac{3}{8}y_3y_4n^3df{n\choose 2}- \frac{\left(f{n\choose 2}\right)^2}{4}y_3n.
\end{align*}
This proves the claim as ${n\choose 2}\le \frac{n^2}{2}$.
	\end{proof}
	
	As the final step, we compare $G_1$ to the iterated balanced blow-up of $C_5$ on $n$ vertices. Note that all induced $C_5$ in $G_1$ either contain one vertex from each $X_i$, or are completely inside one $X_i$. Therefore, by induction on $n=5k+j$, $0\le j\le 4$, 
	we have
	\begin{align}\label{g1}
	\begin{split}
	C(n){n\choose 5}-C(G_1){n\choose 5}\ge& ~k^{5-j}(k+1)^j+(5-j)C(k){k\choose 5}+jC(k+1){k+1\choose 5}\\
	&-\left(\prod_{i=1}^5 y_in + \sum_{i=1}^5 C(y_in){y_in\choose 5}\right).
	\end{split}
	\end{align}
	
	We then wish to show that the balanced iterated blow-up of a $5$-cycle contains more $5$-cycles than $G$, which we do by creating an integer program to bound that difference. In particular, from Claims \ref{count1}, \ref{count2}, and \ref{count3}, we may bound the net gain of $5$-cycles created by removing the funky edges from $G$ to get $G_1$. Then from \eqref{g1}, we may also bound the gain in $5$-cycles going from $G_1$ to the balanced iterated blow-up. This gives an objective function, which is a lower bound on the difference in $5$-cycles going from $G$ to the balanced iterated blow-up. Thus, if our integer program evaluates to a positive number, we know that $G$ cannot possibly be a counterexample. We also include Lemma \ref{numFunky} as bounds in the program. Furthermore, if we examine Claim \ref{count1}, we can see that $f {n \choose 2} \ge dn+1$, as otherwise we would have a negative number of $5$-cycles. Therefore, we solve the following program $(P)$ in the variables $(y_1,y_2,y_3,y_4,y_5,f,d)$, for the fixed $n=5k+j$, $0\le j\le 4$: 
	\begin{align*}
	{\bf (P):}  &
	\text{minimize} \\
	& f{n\choose 2}n^3 \Bigg(y_3 y_4 y_5 -  \frac38d y_3 y_4 - \frac18 f y_3\\
	& -\frac{1}{4}\left(f -\frac{f+d}{n}-\frac{1}{n^2}\right)\left( y_1+y_2+\frac12(y_3+y_4+y_5)\right)
	- \frac{9}{32} \left(d+\frac{2}{n}\right) y_1^2 \Bigg)\\
	& + k^{5-j}(k+1)^j+(5-j)C_5(k){k\choose 5}+jC_5(k+1){k+1\choose 5}\\
	&-\left(\prod_{i=1}^5 y_in + \sum_{i=1}^5 C_5(y_in){y_in\choose 5}\right)\\
	&\text{subject to}\\
	& \sum_{i=0}^5 y_i = 1, \\
	& \sum_{1\le i<j\le 5}y_iy_{j}-f\frac{n-1}{2n} \ge  \frac{2 ( -0.175431374077117 + 8.75407592662244 C(n^*))}{21 C(n^*)},\\
	& f{n\choose 2}\ge dn+1,\\
	& y_{i}\ge y_{i+1}\ge 0  \text{ for } i \in \{1,\ldots,4\}, \\
	& ny_i\in \mathbb{N}.
	\end{align*}
	Looking a bit closer, we quickly see that in an optimal solution, we have 
	that $f=0$ (and we are done) or $f$ is maximized subject to the $y_i$, and that $d$ is maximized subject to $f$, which happens when the funky edges induce a star.
	Then 
	\[
	\frac{2dn+2}{n(n-1)}=f=\sum_{i<j}y_iy_{j}- \frac{2 ( -0.175431374077117 + 8.75407592662244 C(n^*))}{21 C(n^*)},
	\]
	so $(P)$ reduces 
	to a quartic program in the $4$ free variables $y_1,y_2,y_3,y_4$, with all other variables dependent on these four.
	
	We check every $9 \le n \le 1000$, for all possible values of $y_1,y_2,y_3,y_4$, with the help of a computer. It would be feasible to extend this approach a fair bit beyond $n=1000$, but there is no need as our other case easily takes care of these values.
	
	This leads to a list of $14$ possible values of $y_1,y_2,y_3,y_4$ where the objective function is negative, with at most $22$ vertices, we have included the list in the Appendix. Note that each of these may correspond to more than one graph, as $y_1,\dots,y_5$ may not be in the same order as $x_1,\dots,x_5$. However in most cases there are only one or two ways in which the $y_i$ may be matched to the $x_i$ once we consider the symmetry of the $5$-cycle and the two colors. Since the value in the objective function is merely a bound on the difference in the number of $5$-cycles between $H$ and the iterated blow-up of a $5$-cycle, this does not imply that the part sizes will give a counterexample, but rather that we need to check these values separately with more care. 


For this, we first make use of Lemma \ref{numFunky} to bound the number of funky edges for each set of possible values of $x_1,\dots,x_5$. In none of the cases we have to consider more than $6$ funky edges. Then, we consider all locations these funky edges can be in. Each funky edge can be between any of the $10$ pairs $(X_i,X_j)$, giving us at most ${9+k\choose k}$ choices for these pairs of $k$ funky edges, and then we have to consider all possible incidences of the funky edges.

Even if we were to reduce the number of such cases further through the use of symmetries, it would be very unpleasant for a human analysis. But is very easy with the help of the computer, even without any deeper analysis. The location of the funky edges completely determines the color of all edges between the $X_i$.

We do not assign colors to the edges inside the $X_i$ to keep the number of cases manageable. Instead, we count every set of $5$ vertices that could induce a $C_5$ given the right choice of colors inside the $X_i$, even if two such sets would require conflicting colors. We compare this count with the number of $C_5$ in the iterated balanced blow-up of $C_5$, and in all but one case, the iterated blow-up wins.

The only remaining case is $X_1 = X_2 = 3, X_3 = X_4 = X_5 = 1$, with a matching of three funky edges between $X_1$ and $X_2$, see Figure \ref{fig_Final}. This case counts $18$ possible $5$-cycles, $6$ using one vertex from each $X_i$, and $12$ using exactly $2$ of the $3$ funky edges. This is more than the balanced blow-up on $9$ vertices, which contains $16$ $5$-cycles. But here, we can use that the last $12$ of the possible $5$-cycles in this case can be paired into $6$ pairs with conflicting colors on the edges inside $X_1$ and $X_2$, so that at most one in each pair can actually be a $5$-cycle. Therefore, no  coloring of the $6$ edges inside $X_1$ and $X_2$ can create more than $12$  $5$-cycles.

\begin{figure}[H]
\begin{center}
\newcommand{\base}{
\draw \foreach \x  in {0,...,4}{
(90+\x*72:1) node[vtx](x\x){}
(90+\x*72:1.3) coordinate(y\x)
(90+\x*72:1.7) coordinate(z\x)
}
(x0)--(x1)--(x2) (x3)--(x4)--(x0)
;
}
\begin{tikzpicture} \base \draw (y2) node[vtx]{} (y3) node[vtx]{}  (z2) node[vtx]{} (z3) node[vtx]{}  (x1)--(y2) (y3)--(x4)  (x1)--(z2) (z3)--(x4) (x2)--(y3) (y2)--(z3) (z2)--(x3) (x2)--(z3) (y2)--(x3) (z2)--(y3)  ;  \end{tikzpicture}
\end{center}
\caption{The final remaining case with $X_1 = X_2 = 3, X_3 = X_4 = X_5 = 1$. Only red edges known to be there are shown.}\label{fig_Final}
\end{figure}
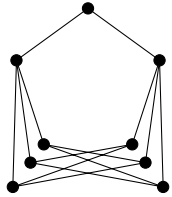

	\begin{case} \label{large}
		$n\ge 1000$:
	\end{case}

	As we are dealing with infinitely many values of $n$, we first establish a common bound for $C(G^*)$ for all $n \ge 1000$.
	
	\begin{prop}
		For $n \ge 1000$, $C(G^*) > 0.0384609$.
	\end{prop}

	\begin{proof}
Since we know that $C(H) \ge C(n)$ and thus $C(G^*) \ge C(n^*)$, it suffices to bound $C(n^*) > 0.0384609$ for $n \ge 1000$. 
Note that from $C(n)\ge \frac1{26}$, it follows that
\[
C(n^*)>\frac{(n-1)(n-2)(n-3)(n-4)}{n^4}C(n)\ge\frac{(n-1)(n-2)(n-3)(n-4)}{26n^4} .
\]
For $n\ge 610000$, this quantity is larger than $0.0384609$, so one way to show the proposition is to explicitly calculate $C(n^*)$ for all $n\le 610000$, and then use this observation.

At this point violating our philosophy of not arguing facts by hand that can easily be checked by the computer, we give
a slightly less computational proof. We only check that the claim is true for $n \le 5000$ by explicit computation, and then argue by induction. Let $n\ge 1000$, and $C=\min\{C(n^*),C((n+1)^*)\}$, then for $0\le i\le 4$,
		\begin{align*}
		C((5n+i)^*)=&~120\left(\frac{n}{5n+i}\right)^{5-i}\left(\frac{n+1}{5n+i}\right)^{i} 
		                    +\frac{(5-i)n}{5n+i}\left(\frac{n}{5n+i}\right)^4C(n^*)\\
		&~+\frac{i(n+1)}{5n+i}\left(\frac{n+1}{5n+i}\right)^4C((n+1)^*)\\
		\ge&~ 120\left(\frac{n}{5n+i}\right)^{5-i}\left(\frac{n+1}{5n+i}\right)^{i}+\left(\frac{n}{5n+i}\right)^4C\\
		\ge&~ \left(\frac{1}{5n+i}\right)^5\left( 120(n^5+in^4)+(5n^5+in^4) C  \right)\\
		>&~ \left(\frac{n}{5n+i}\right)^5\left( 120+5C+120\frac{i}{n}\right).
		\end{align*}
		Now for $n = 1000, 0 \le i \le 4$ this value is larger than $0.0384609$. We also know that,
		\[
		\frac{\partial}{\partial n}\left(\frac{n}{5n+i}\right)^5\left( 120+5C+120\frac{i}{n}\right)=5 i n^3 \frac{ 5 C n+96i}{(5 n+i)^6} > 0.
		\]
		Therefore, as for fixed $i$ we know that $C((5n+i)^*)$ is increasing with respect to $n$, and since $C(n^*) > 0.0384609$ for $1000 \le n \le 5000$, we have the desired result. 
	\end{proof}

\begin{subcase}
$d\le 0.2$:
\end{subcase}
	We first assume that $d$, the normalized average funky degree sum of funky edges, is small. 
	We use the same process as before, where we flip all funky edges and then compare the number of $5$-cycles. 
%

	Consider the following program $(P')$ with $C=0.0384609$ for any fixed $d$. It is derived from $(P)$ by first dividing the objective function by $f{n\choose 2}n^3$, and then using $n=1000$ or $n\to\infty$ depending on which is yielding a lower objective function. Also, we skip the last step of balancing the parts for an easier objective function. We account for this in Claim \ref{balance}.
	\begin{align}
	{\bf (P'):}&
	\text{minimize} \nonumber \\
	& y_3 y_4 y_5 -  \frac38dy_3 y_4 - \frac18 f y_3
	-\frac{1}{4}f \left( y_1+y_2+\frac12(y_3+y_4+y_5)\right)
	- \frac{9}{32} d y_1^2-\frac{9}{16\times 1000}y_1^2\label{obj}\\
	&\text{subject to}\nonumber \\
	& \sum_{i=1}^5 y_i = 1, \\
	& \sum_{1\le i<j\le 5}y_iy_{j}-f\frac{1000-1}{2\times 1000} \ge \frac{2 ( -0.175431374077117 + 8.75407592662244 C)}{21 C},\label{yibound}\\
	& f>0,\nonumber \\
	& y_i\ge y_{i+1} \geq 0 \text{ for } i \in \{1,\ldots,4\}. \label{yiorder}
	\end{align}
	The objective function \eqref{obj} decreases for increasing $d$ and $f$. Consequently, we fix $d=0.2$. We know that $f$ is maximized in~\eqref{yibound}  for $y_1=y_2=y_3=y_4=y_5=0.2$, and we fix $f$ at this maximum in \eqref{obj}.  At the same time, the bound on the $y_i$ derived from~\eqref{yibound} is weakest for $f=0$, so we will use $f=0$ when applying this bound.
	
This leaves us with a continuous cubic program in the four variables $y_1,y_2,y_3,y_4$, with dependent variable $y_5=1-y_1-y_2-y_3-y_4$.	Instead of trying to solve this program, we discretize to find a lower bound greater than zero, the desired contradiction.

For any grid point $(t_1,t_2,t_3,t_4)$ and some $\eps>0$, we consider the cell $\prod [t_i,t_i+\eps]$. Note that this implies a range of $[t_5-4\eps, t_5]$ for the size of the smallest part if we set $t_5=1-t_1-t_2-t_3-t_4$. We check if the cell contains a point $(y_1,y_2,y_3,y_4)$ satisfying \eqref{yiorder}. If this is the case, then we check if there may be a point (not necessarily the same) in the cell satisfying \eqref{yibound} by computing generously $t_5(1-t_5)+\sum_{1\le i<j\le 4}(t_i+\eps)(t_j+\eps)$. If the answer is positive, we lower bound \eqref{obj} in the box by computing
\begin{align}
(t_3+\eps)(t_4+\eps) (t_5-4\eps) -  \frac38d(t_3+\eps)( t_4+\eps) - \frac18 f (t_3+\eps)\hspace{3cm}\nonumber\\
	-\frac{1}{4}f \left( t_1+t_2+\frac12(t_3+t_4+t_5)+\eps\right)
	- \frac{9}{32} d (t_1+\eps)^2-\frac{9}{16000}(t_1+\eps)^2.\label{obj2}
\end{align}
Every term in this sum but possibly the first is easily seen to be a lower bound for the corresponding term in \eqref{obj} over all values of $(y_1,y_2,y_3,y_4)$ in the cell. The first term is a lower bound over all values satisfying \eqref{yiorder}.
%
%
	
	To reduce the number of points to check, we include a few additional considerations. First, note that from \eqref{yibound}, we can get the additional constraint that $0.166 \le y_i \le 0.234$. Secondly, rather than fixing some $\eps > 0$ and checking all cells, we iteratively refine the mesh only where needed. This allows us to have a more refined search, as some cells in our feasible region will clearly produce positive objective values. We begin by initializing with a single cell with $t_i = 0.166$ for $i \in [4]$ and $\eps = 0.234-0.166$. Then every time when \eqref{obj2} evaluates to $<0.0001$ (to allow for rounding errors), we halve $\eps$ and create $2^4$ new points depending on whether $t_i$ remains the same or $t_i = t_i + \frac{\eps}{2}$. These $16$ new cells are added to a stack. Cells in the stack are evaluated one by one, each time either removing it if  \eqref{obj2} evaluates greater than $0.0001$, or removing it and adding $16$ new cells to the stack. 
	
	 The program runs in a few minutes on a laptop, and makes around $1.8 \cdot 10^{6}$ calls to the objective function \eqref{obj2}. Furthermore, the stack never contains more than $100$ elements, meaning that we never have to iterate too far into one specific area of the feasible region. Note that with more computational effort, this program could also yield a contradiction for some larger value of $d$. But $d=0.2$ more than suffices for the next case. 
	
\begin{subcase}
$d > 0.2$:
\end{subcase}

We now show that we can not have $d > 0.2$ by looking at a single vertex with maximum funky degree. Let $v$ be such a vertex with maximum funky degree $d_f(v)=\Delta_f>0.1n$. Note that in the remainder of the proof all $5$-cycles we consider contain $v$, and we will not point this out every time. We will use a rule to move $v$ to one of the parts $X_1,\ldots,X_5$, and flip all resulting funky edges incident to $v$ to create a graph $G_1$. We then bound the number of $5$-cycles created and destroyed and show that we have more $5$-cycles in $G_1$, our desired contradiction. Without loss of generality assume that $v \in X_1$ at the beginning.
	
Let $r_in$ and $b_in$ be the numbers of red and blue neighbors of $v$ in $G$ in $X_i$, respectively. 
As the partition into the $X_i$ maximizes the number of non-funky edges, moving $v$ to some new part cannot increase this number. Therefore,
	
	\[
	r_2+b_3+b_4+r_5\ge \max\{r_1+b_2+b_3+r_4,r_3+b_4+b_5+r_1,r_4+b_5+b_1+r_2,r_5+b_1+b_2+r_3\}.
	\]
	
	Furthermore as $f > 0.2$,
	
	\[
	b_2+r_3+r_4+b_5=\frac{d_f(v)}{n}> 0.1.
	\]
	
	For some $1\le i\le 5$, move $v$ to  $X_i$, and flip all resulting funky edges incident to $v$ after the move to create the graph $G_1$. We bound the numbers of $5$-cycles containing $v$ in $G$ and $G_1$, and depending on these bounds we choose which $X_i$ we move $v$ to. As no edges from $v$ to this $X_i$ are flipped, the number of $5$-cycles inside $X_i$ is not affected by the flip. In $G_1$, there are at least
	\begin{align}\label{bigcyc}
	\frac{x_1x_2x_3x_4x_5}{x_i}n^4-f{n\choose 2}n^2\max_{|\{i,j,\ell\}|=3}x_jx_\ell
	\end{align}
	$5$-cycles which have at least one vertex outside of $X_i$. To see this, we simply pick one vertex for every single part not $X_i$. The only reason they would not form a $C_5$ in $G_1$ is if there was a funky edge between two of these four vertices. Every funky edge then destroys at most $n^2\max_{|\{i,j,\ell\}|=3}x_jx_\ell$ $5$-cycles of this form. 
	
	We choose $i$ to maximize \eqref{bigcyc}, so let 
	\[
	M_1:=\max_i\left\{ \frac{x_1x_2x_3x_4x_5}{x_i}n^4-f{n\choose 2}n^2\max_{|\{i,j,\ell\}|=3}x_jx_\ell \right\}.
	\]
	
	That is, $M_1$ is a lower bound on the number of $5$-cycles not entirely in $X_i$ in $G_1$, and we wish to compare this to the number of $5$-cycles in $G$. We first bound the number of $5$-cycles in $G$ in which all funky edges are incident to $v$. In particular, the remaining four vertices must induce a $P_4$, so they must either all lie in the same $X_j$, or in four different $X_j$s. The number of such $5$-cycles containing a vertex outside of $X_i$ is thus at most
	\[
	M_2:=\left(r_1b_2b_3r_4+r_2b_3b_4r_5+r_3b_4b_5r_1+r_4b_5b_1r_2+r_5b_1b_2r_3+\tfrac{1}{16}(r_2^2b_2^2+r_3^2b_3^2+r_4^2b_4^2+r_5^2b_5^2)\right)n^4.
	\]
	Let us now bound the number of $5$-cycles in $G$ containing a funky edge not incident to $v$.  There are at most
	\[
	f{n\choose 2}\frac14n^2
	\]
	such cycles, as we can first pick some funky edge, and then select two other vertices (see Lemma~\ref{lem}). This however over counts all cycles which contain more than one funky edge not incident to $v$.
	To get a better bound, we will now bound the number of cycles which contain exactly one funky edge $uw$ not incident to $v$.
	There are ten different cases depending on the location of $uw$. Since all cases are symmetric by rotation or a color switch, we only have to analyze one case in detail. 
	
	Let us assume that $u\in X_1$, $w\in X_2$, so $uw$ is a blue funky edge.
	Let $x,y$ be the remaining $2$ vertices of a $C_5$. There are three cases depending on the colors of $uv$ and $vw$ (they cannot both be blue).
	If $uv$ and $vw$ are red, then $xv$ and $yv$ are blue, and we may assume (by symmetry) that $xu$ and $wy$ are the remaining two blue edges of the $C_5$. 
	Then $x\in X_1,y\in X_2$, or $x\in X_1,y\in X_5$, or $x\in X_3,y\in X_2$, as otherwise there would be more funky edges.
	
	If $uv$ is blue and $vw$ is red, then we may assume that $vuwxyv$ is the blue $C_5$. Then
	$x\in X_5,y\in X_2$, or $x\in X_2,y\in X_2$. Finally, if $uv$ is red and $vw$ is blue, and $vwuyxv$ is the blue $C_5$, then
	$x\in X_1,y\in X_3$, or $x\in X_1,y\in X_1$.
	Altogether, the number of $5$-cycles containing $\{u,v,w\}$ and no other funky edge not incident to $v$ is at most
	\[
	\max\{b_1b_2+b_1b_5+b_3b_2,r_5b_2+r_2b_2,b_1r_3+b_1r_1\}n^2.
	\]
	With ten choices for the sets of $\{u,w\}$, this maximum is extended to a maximum of $30$ terms:
	\[
	M_3:= \max\left.
	\begin{cases}
	~b_1b_2+b_1b_5+b_3b_2,~r_5b_2+r_2b_2,~b_1r_3+b_1r_1,\\
	~b_2b_3+b_2b_1+b_4b_3,~r_1b_3+r_3b_3,~b_2r_4+b_2r_2,\\
	~b_3b_4+b_3b_2+b_5b_4,~r_2b_4+r_4b_4,~b_3r_5+b_3r_3,\\
	~b_4b_5+b_4b_3+b_1b_5,~r_3b_5+r_5b_5,~b_4r_1+b_4r_4,\\
	~b_5b_1+b_5b_4+b_2b_1,~r_4b_1+r_1b_1,~b_5r_2+b_5r_5,\\
	~r_1r_3+r_5r_3+r_1r_4,~b_4r_3+b_3r_3,~b_5r_1+b_1r_1,\\
	~r_2r_4+r_1r_4+r_2r_5,~b_5r_4+b_4r_4,~b_1r_2+b_2r_2,\\
	~r_3r_5+r_2r_5+r_3r_1,~b_1r_5+b_5r_5,~b_2r_3+b_3r_3,\\
	~r_4r_1+r_3r_1+r_4r_2,~b_2r_1+b_1r_1,~b_3r_4+b_4r_4,\\
	~r_5r_2+r_4r_2+r_5r_3,~b_3r_2+b_2r_2,~b_4r_5+b_5r_5
	\end{cases}\right\}.
	\]
	Therefore, we get the following upper bound for the 
	number of $5$-cycles containing a funky edge not incident to $v$ after we adjust for double counts:
	\begin{equation}\label{eq:P''}
	 f {n \choose 2} n^2\frac{1}{2} \left(\frac{1}{4}-M_3 \right) + M_3f {n \choose 2} n^2.
	\end{equation}
	The first term bounds cycles with more than one funky edge not adjacent to $v$, where the $\frac12$ comes from the fact that $ f {n \choose 2} n^2$ at least double counts these $5$-cycles. The second term bounds the number of $5$-cycles with exactly one funky edge not adjacent to $v$. We then create a mathematical program $(P'')$, we wish to lower bound, with \eqref{eq:P''} as our objective function. We also include the same bounds coming from Lemma \ref{numFunky} as well. 
	
	\begin{align*}
	{\bf (P''):} 
	&\text{minimize} \\
	& n^{-4}\left(M_1-M_2-\left(\frac18+\frac12 M_3\right)f{n\choose 2}n^2\right) \\
	&\text{subject to}\\
	& \sum_{i=1}^5 x_i = 1, \\
	& x_i=r_i+b_i,\\
	& \sum_{1\le i<j\le 5}x_ix_{j}-f\frac{n-1}{2n} \ge \frac{2 ( -0.175431374077117 + 8.75407592662244 \,C)}{21 \,C},\\
	& f>0\\
	& r_i,b_i \geq 0 \text{ for } i \in \{1,\ldots,4\}. \\
	\end{align*}
	
	The factor of $n^{-4}$ in the objective function is for normalization, and cancels many terms. We fix $f$ at its maximum of $\frac{2000}{999}\left( 10\times 0.2^2-\frac{2 ( -0.175431374077117 + 8.75407592662244 C)}{21 C}\right)$. The objective function grows with $n$, so we fix $n=1000$. 
	
	Similar to how we solved $(P')$, we cover the feasible region by an $\varepsilon$-grid in the nine variables $x_2, x_3, x_4, x_5$, $r_1, r_2, r_3, r_4, r_5$ with dependent variables $x_1,b_1,b_2,b_3,b_4,b_5$, and replace every variable in each term of the function by its maximum or minimum in each grid cell to bound the function. We also introduce the same constraints of $0.166 \le x_i \le 0.234$ as in $(P')$ to help speed up computation. We then use the same technique of reducing $\varepsilon$ by a factor of $\frac12$ each iteration, creating now $2^9$ new cells for the independent variables. 
	It turns out that $(P'')$ requires even less computation than $(P')$ running in less than a minute with fewer than $1,000$ calls to the objective function, despite the fact that the discretization creates more cells at each iteration. 
	
This proves that there are no funky edges, so $G$ is a blow-up of $C_5$. It remains to show that the blow-up is balanced, then Theorem~\ref{thm_main} follows by induction.

\begin{claim}\label{balance}
The extremal graph $G$ is a balanced blow-up of $C_5$.
\end{claim}
\begin{proof}
We proceed by induction on $n$. We assume the statement is true for all smaller values. Then the number of $5$-cycles in an iterated blow-up with parts of sizes $n_1,n_2,n_3,n_4,n_5$ is at most
\[
n_1n_2n_3n_4n_5+C(n_1){n_1\choose 5}+C(n_2){n_2\choose 5}+C(n_3){n_3\choose 5}+C(n_4){n_4\choose 5}+C(n_5){n_5\choose 5}.
\]
As this quantity is symmetric in the $n_i$, we may assume from now on that $n_1\ge n_2\ge n_3\ge n_4\ge n_5$.
For $n\le 1000$, we explicitly compute these quantities for all partitions $n=n_1+n_2+n_3+n_4+n_5$, and verify that the lemma is true.

For $n>1000$, assume that $n_1-n_5\ge 2$. Note that \eqref{yibound} again implies that $0.166n \le n_5<n_1 \le 0.234n$.
Let $v \in X_1$ where the number of $5$-cycles $C_5^v$ containing $v$ is minimized  over the vertices in $X_1$.
Let $w \in X_5$ where the number of $5$-cycles $C_5^w$ containing $w$ is maximized  over the vertices in $X_5$.
The number of $5$-cycles containing both $v$ and $w$ is $n_2n_3n_4$.
If $C_5^w-n_2n_3n_4-C_5^v>0$, we can increase the number of $5$-cycles
by replacing $v$ by a copy of $w$, contradicting the extremality of $G$.

As $C(n)$ is non-increasing, we have
\[
0.04086 \ge C(166)\ge C(n_5)\ge C(n_1).
\]
Therefore, we have
\begin{align*}
C_5^w-n_2n_3n_4-C_5^v&\ge \frac{C(n_5){n_5\choose 5}}{n_5}+n_1n_2n_3n_4-n_2n_3n_4-\frac{C(n_1){n_1\choose 5}}{n_1}-n_2n_3n_4n_5\\
&= \frac{C(n_5){n_5-1\choose 4}-C(n_1){n_1-1\choose 4}}{5}+(n_1-n_5-1)n_2n_3n_4\\
&\ge  \frac{C(n_5)\left({n_5-1\choose 4}-{n_1-1\choose 4}\right)}{5}+(n_1-n_5-1)n_5^3\\
&\ge  \frac{C(166)\left(n_5^4-n_1^4\right)}{5!}+(n_1-n_5-1)n_5^3\\
&=  \frac{C(166)}{ 5!}(n_5-n_1)\left( n_5^3+n_5^2n_1+n_5n_1^2+n_1^3\right)+(n_1-n_5-1)n_5^3\\
&\ge  \frac{4C(166)}{ 5!}(n_5-n_1)n_1^3+\frac12(n_1-n_5)n_5^3\\
&= \frac12(n_1-n_5)\left( n_5^3-\frac{8C(166)}{ 5!}n_1^3\right)\\
&\ge \left( 0.166^3-\frac{8C(166)}{ 5!}\, 0.234^3\right)n^3\\
&>0,
\end{align*}
a contradiction.
\end{proof}
This proves Theorem~\ref{thm_main}.


\section{Proof of Lemma~\ref{funky}}\label{sec:lem}

We use flag algebras to show a slightly stronger statement that every sufficiently large graph $G$ with $C(G) \geq 0.03$
satisfies
	\[ C^{\bullet \bullet}(G) \ge -0.175431374077117 + 8.75407592662244 \, C(G) .\]
This type of inequality was used by Lidick\'y and Pfender~\cite{LP18} when solving the  Pentagon problem of Erd\H{o}s for small graphs.
The flag algebra method has been developed by Razborov~\cite{Razborov}, and has seen numerous applications such as ~\cite{GHV2019,CLP2020,KLTP2019,PST2019,GKV2016,BT2014,CKPSTY2013}. 
We assume the reader is familiar with the method and describe only a brief outline of the calculation rather than developing the entire theory and terminology. 
A description of the method when applied to graphs is available from several sources~\cite{PST2019,BDL2020}.
The calculation is computer assisted, and the program we used can be downloaded from the \href{https://arxiv.org/abs/2102.06773}{arXiv} version of this paper or \oururl.

Let $\varphi$ correspond to a convergent sequence of graphs $(G_i)_{i > 0}$.
For a  graph $H$ we denote by $\varphi(H)$ the limit of densities of $H$ in $G_i$ as $i$ tends to infinity.
Since $\varphi$ is actually a homomorphism to $\mathbb{R}$, it naturally extends to formal linear combinations of graphs.
The following inequalities are satisfied for any $\varphi$ and $\ell \ge 7$.
\begin{align*}
\varphi(C^{\bullet \bullet}) &= \varphi\left(  \sum_{F \in \mathcal{F}_\ell}  c_F F   \right) =   \sum_{F \in \mathcal{F}_\ell}  c_F \varphi(F)\\
0 &\geq \sum_{F \in \mathcal{F}_\ell}  -a_F \varphi(F) \\
0 &\geq -  \varphi\left(\sum_{\sigma} \llbracket x_\sigma^T M_\sigma x_\sigma \rrbracket_\sigma \right)
=   - \sum_{F \in \mathcal{F}_\ell}  e_F  \varphi(F)   
\end{align*}
where $\mathcal{F}_{\ell}$ are all graphs on $\ell$ vertices up to isomorphism, $c_F$ is the sum of densities of graphs in $C^{\bullet \bullet}$ in $F$, $a_F$ is any non-negative real number, $\sigma$ is a type, $x_\sigma$ is a vector of $\sigma$-flags, $M_\sigma$ is a positive semidefinite matrix,  $\llbracket  . \rrbracket_\sigma$ is the unlabeling operator, and $e_F$ are some real coefficients  depending on $F$, $x_\sigma$, and $M_\sigma$.
We also add the following constraint
\[
 0 \geq s (0.03 - \varphi(C_5)) = 0.03\cdot s -   \sum_{F \in \mathcal{F}_\ell} s b_F  \varphi(F) ,
\]
where $s$ is a positive real number and $b_F$ is the density of $C_5$ in $F$.
Combining all inequalities,  and using  $ \sum_{F \in \mathcal{F}_\ell} \varphi(F) =1$  gives
\begin{align*}
\varphi(C^{\bullet \bullet}) &\geq
\sum_{F \in \mathcal{F}_\ell}  (c_F-a_F-e_F-sb_F) \varphi(F) + 0.03\cdot s\\
&\geq
\min_{F \in \mathcal{F}_\ell}  (c_F-a_F-e_F-sb_F) \left( \sum_{F \in \mathcal{F}_\ell} \varphi(F) \right)+ 0.03\cdot s\\
&= \min_{F \in \mathcal{F}_\ell}  (c_F-a_F-e_F-sb_F) + 0.03\cdot s.
\end{align*}
Notice that the expression depends on positive semidefinite matrices $M_\sigma$ and the value of $s$.
We may optimize this lower bound using semidefinite programming software.

We use CSDP~\cite{csdp}
and $\ell=8$, and obtain a numerical solution.
We round the numerical solution to an exact rational solution using SageMath~\cite{sagemath} and obtain the following:
\begin{align*}
\varphi(C^{\bullet \bullet}) \geq&~   -\frac{175431374077116112876105446118032690611106}{1000000000000000000000000000000000000000000} \\ &~ + 0.03 \cdot\frac{8754075926622441195046069111932573299245056}{1000000000000000000000000000000000000000000} \\
 >&~   -0.175431374077117 + 0.03 \cdot 8.75407592662244.
\end{align*}
Notice that the lower bound does not depend on $\varphi$.
Therefore, assuming $\varphi(C_5) \geq 0.03$, we get a valid bound by replacing $0.03$ in this equation with the the density of $C_5$. 
Note further that the numerical approximation is strictly smaller than the rational solution.
This makes the result valid for every sufficiently large graph $G$, since flag algebra calculations on graphs instead of graph limits have error $o(1) = O(\frac{1}{n})$, where $n$ is the number of vertices of $G$.


\section{Further Directions}
As mentioned above, we know that $C_6$ and the net $N$ on $6$ vertices have (F3). For $N$, we know that it does not have (F5) as, similarly to $C_5$, there is a small extremal graph which is not a blow-up of $N$. For $C_6$, we are not aware of such an example, and our methods may be successful here.

As another direction, the notion of fractalizers directly translates to directed graphs.
It is easy to direct the edges in an iterated balanced blow-up of $C_5$ so that every induced copy of $C_5$ becomes a directed $\vec{C}_5$. This is not possible for the M\"obius ladder on $8$ vertices, so we get the following theorem as an immediate corollary of Theorem~\ref{thm_main}.
\begin{thm}
$\vec{C}_5$ is a fractalizer.
\end{thm}
From related unpublished work~\cite{HLPV202X}, we know that $\vec{C}_4$ also has (F3), and we conjecture that it in fact fractalizes. 
\begin{conj}
For all $k\ge 4$, $\vec{C}_k$ is a fractalizer.
\end{conj}
For $\vec{C_3}$, the iterated balanced blow-up asymptotically achieves the maximum number of $\vec{C}_3$. Nevertheless, for many values of $n$, it fails to be extremal. This stems from the folklore fact that the number of $\vec{C_3}$ is maximized if and only if the graph is a regular (or near regular for even $n$) tournament. For an infinite number of values of $n$, including all values of the form $n=6k\pm 1$, the iterated balanced blow-up of $\vec{C_3}$ has vertices which differ in out-degree by at least $2$. So $\vec{C_3}$ has (F1) but not (F2).
%
%

\section*{Acknowledgement}
The authors thank an anonymous referee for comments on improving this manuscript.
This work used the computing resources at the Center for Computational Mathematics, University of Colorado Denver, including the Alderaan cluster, supported by the National Science Foundation award OAC-2019089.

\bibliographystyle{plainurl}
\bibliography{references}

\section*{Appendix}
The following is a list of the $23$ different values of $x_1,\dots,x_5$ such that program $(P)$ has a negative objective value. Note that $(P)$ produces values for $y_1,\dots,y_5$, which may have a different ordering than $x_1,\dots,x_5$. We therefore list all possible values of $x_1,\dots,x_5$ based on each $y_1,\dots,y_5$, up to isomorphism. \\

\noindent (1,1,1,3,3)
(1,3,1,1,3)
(1,1,2,2,3)
(1,2,3,2,1)
(1,2,3,1,2)
(1,2,2,1,3)
(1,2,2,2,2)
(2,2,2,2,3)
(2,2,2,2,4)
(2,2,2,3,3)
(2,3,2,2,3)
(1,3,3,3,3)
(2,2,2,3,4)
(2,2,3,3,3)
(2,3,2,3,3)
(2,3,3,3,3)
(3,3,3,3,4)
(3,3,3,4,4)
(3,4,3,3,4)
(3,3,4,4,4)
(3,4,3,4,4)
(4,4,4,5,5)
(4,5,4,4,5)

\end{document}